\documentclass[letterpaper]{article}
\usepackage[utf8]{inputenc}
\usepackage{enumitem, comment}

\title{\vspace{-1cm}On the edge-density of the Brownian co-graphon\\ and common ancestors of pairs in the CRT
}

\author{Guillaume Chapuy%
\thanks{Université Paris Cité, CNRS, IRIF, F-75013, Paris, France
Email:~{\tt guillaume.chapuy@irif.fr}.
 G.C. acknowledges funding from the grants ANR-19-CE48-0011 ``COMBINÉ'' and ANR-18-CE40-0033  ``Dimers''. Thanks to Lucas Gerin for a talk at IRIF presenting the results of~\cite{bbfgmp} and for interesting discussions. 
}
}

\date{\today}

\usepackage{graphicx, enumitem}
\usepackage{amsmath,amsfonts,amssymb}
\usepackage{amsthm, xcolor}
\usepackage{comment, fullpage, placeins,hyperref}

\usepackage[margin=7mm]{caption}

\theoremstyle{theorem}
\newtheorem{theorem}{Theorem}
\newtheorem{proposition}[theorem]{Proposition}
\newtheorem{corollary}[theorem]{Corollary}
\newtheorem{lemma}[theorem]{Lemma}
\newtheorem{claim}[theorem]{Claim}

\numberwithin{theorem}{section}

\theoremstyle{definition}

\newtheorem{remark}[theorem]{Remark}

\newcommand{\Cat}[1]{\mathcal{C}_{#1}}

\begin{document}

\maketitle

\begin{abstract} 
	Bassino et al.~\cite{bbfgmp}
	have shown that uniform random co-graphs (graphs without induced $P_4$) of size $n$ converge to a certain non-deterministic graphon. The edge-density of this graphon is a random variable $\Lambda \in [0,1]$ whose first moments have been computed by these authors.
	The first purpose of this note is to observe that, in fact, these moments can be computed by a simple recurrence relation. 

	The problem leads us to the following question of independent interest: given $k$ i.i.d. uniform pairs of points $(X_1,Y_1),$ $\dots$, $(X_k,Y_k)$ in the Brownian CRT, what is the size $S_k$ of the set $\{X_i\wedge Y_i, i=1\dots k\}$ formed by their pairwise last common ancestors?
We show that ${S_k}\sim c{\sqrt{k}\log k}$ in probability, with  $c=(\sqrt{2\pi})^{-1}$.

	The method to establish the recurrence relation is reminiscent of Janson's computation of moments of the Wiener index of (large) random trees. The logarithm factor in the convergence result comes from the estimation of Riemann sums in which summands are weighted by the integer divisor function -- such sums naturally occur in the problem.

	We have not been able to analyse the asymptotics of moments of $\Lambda$ directly from the recurrence relation, and in fact our study of $S_k$ is independent from it.
	Several things remain to be done, in particular we only scratch the question of large deviations of $S_k$, and the precise asymptotics of moments of $\Lambda$ is left open.




\end{abstract}

\section{Introduction and main results}

A \emph{cograph}  of size $n$ is a graph on $[n]$ (no loops nor multiple edges) which avoids the path $P_4$ as an induced subgraph. The class of cographs has a simple recursive decomposition in terms of \emph{join} or \emph{disjoint union} operations -- this a special case of the modular decomposition of a graph. This tree-like structure enabled the authors of~\cite{bbfgmp} to show that cographs of size $n$ taken uniformly at random converge, when $n$ goes to infinity, to a certain non-deterministic graphon that they denote $\mathbf{W}^{1/2}$. We refer to~\cite{lovasz:largeNetworks} for basics on graphons and their convergence. 

The present note requires no knowledge on graphons, and results will be intelligible in terms of the Brownian excursion or random trees, but let us still describe quickly the graphon $\mathbf{W}^{p}$, which is defined for any $p\in [0,1]$. Let $\mathbf{e}$ be a normalized Brownian excursion on $[0,1]$. For $x,y\in [0,1]$, $x\neq y$, we let $m^{\mathbf{e}}_{x,y}$ be a value $t$ in the interval $[\min(x,y),\max(x,y)]$ at which the minimum of $\mathbf{e}_t$ is attained. This value is uniquely defined almost surely for almost all $x,y$, which is good enough for our purposes. Now, assume that for each $t\in [0,1]$ which is one of these values (there are countably many, a.s.), one is given a Bernoulli random variable $B_t$ of parameter $p$ (so $B_t$ is equal to 0 or 1 with probability $(1-p)$ and $p$ respectively), such that these variables are altogether independent, and independent of $\mathbf{e}$.
Then the (random) graphon $\mathbf{W}^{p}$ is defined \cite[Definition 4.2]{bbfgmp} as the (graphon equivalence class of the) (random)  function
\begin{align*}
	\mathbf{w}^p:	[0,1]^2 &\longrightarrow \{0,1\} \\
	(x,y) &\longmapsto B_{m^{\mathbf{e}}_{x,y}}.
\end{align*}
Readers unfamiliar with graphons can just think of $\mathbf{w}^p$ as the "continuum adjacency matrix" which is to the graphon $\mathbf{W}^p$ what the adjacency matrix would be to a finite graph. 

In term of the Brownian Continuum Random Tree (CRT, see e.g.~\cite{legall:treeSurvey}) we may think of this object as follows. The Brownian excursion $\mathbb{e}$ defines a continuum tree $\mathcal{T}_\mathbf{e}$ together with a canonical projection $\pi_{\mathbf{e}}: [0,1] \rightarrow \mathcal{T}_\mathbf{e}$. We assume that all the "nodes" of $\mathcal{T}_\mathbf{e}$ are decorated by a Bernoulli variable of parameter~$p$. Then for $x,y \in [0,1]$, the function $\mathbf{w}^p$ associates to the pair of vertices $x,y$ of $\mathcal{T}_\mathbf{e}$ the Bernoulli variable decorating their last common ancestor $x\wedge y := \pi_{\mathbf{e}}(m^\mathbf{e}_{x,y})$. Of course, different pairs of vertices can share the same last common ancestor, which creates dependencies and makes the model interesting.

The authors of~\cite{bbfgmp} proved that, if $K_n$ is a  cograph on $[n]$ taken uniformly at random, then
$$
K_n \longrightarrow \mathbf{W}^{1/2},
$$
in law, for the topology of graphon convergence. In particular, for any fixed graph $H$, one has the convergence (in law) of subgraph densities: 
$$
\mathrm{Dens}(H,K_n) \longrightarrow \mathrm{Dens}(H,\mathbf{W}^{1/2}).
$$
In this note, we are interested in the case where $H$ is a single edge. We consider the case of arbitrary $p$ (which at the discrete level is related to a biased model of random cographs, see~\cite{bbfgmp}). We use the shortcut notation 
$$\Lambda_p:= \mathrm{Dens}(\mathrm{edge},\mathbf{W}^{p}):=\int_{[0,1]^2}\mathbf{w}^{p}.$$
Thus $\Lambda_{1/2}$ is the (random) edge density of the graphon $\mathbf{W}^{1/2}$, or equivalently the limit in distribution of the random variable $|E(K_n)|/{n\choose 2}$.

Moments of $\Lambda_p$ can be studied as follows.
Let $X_1,Y_1,\dots,X_k,Y_k$ be i.i.d. points chosen according to the Lebesgue measure on the CRT $\mathcal{T}_\mathbf{e}$, and let $Z_i=X_i \wedge Y_i$ be the last common ancestor of the $i$-th pair.
Let $B_i$ be the Bernoulli variable sitting at $Z_i$.
By definition of the graphon $\mathbf{W}^p$ (see~\cite{bbfgmp} for details), one has 
\begin{align}\label{eq:deflambdak}
\mathbf{E}\Lambda_p^k = \mathbf{P}(B_1=\dots=B_k=1).
\end{align}
Equivalently, 
$\mathbf{E}\Lambda_p^k = \mathbf{E}p^{S_k},$
where 
$$S_k := |\mathcal{S}_k| \ \ , \ \ \mathcal{S}_k:=\{Z_1,\dots,Z_k\}.$$
In words,  $S_k$ counts the number of distinct last common ancestors of the pairs $(X_1,Y_1)$, $\dots$, $(X_k,Y_k)$, i.e.  the number of distinct Bernoulli variables appearing in~\eqref{eq:deflambdak}. See Figure~\ref{fig:rec}-Left.

Our first observation is the following formula. We use the notation $(2n-1)!!=(2n-1)(2n-3)\dots 1$ for the double factorial, with $(-1)!!:=1$.
\begin{theorem}\label{thm:main}
	Let $a_k(p) := \mathbf{E}\Lambda_p^k= \mathbf{E}p^{S_k}$. Then one has $a_0(p)=1$ and, for $k\geq 1$,
\begin{align}\label{eq:mainrec}
	a_k(p)
	=&
	\sum_{k_1+k_2+k_3=k, k_i\geq 0\atop{k_1+k_3>0 \mbox{\tiny ~and~} k_2+k_3>0}}
	a_{k_1}(p)a_{k_2}(p)\frac{k!}{k_1!k_2!k_3!} 
	\frac{(4k_1+2k_3-3)!!(4k_2+2k_3-3)!!2^{k_3-1}}{(4k-3)!!} p^{\mathbf{1}_{k_3\neq 0}}.
\end{align}
\end{theorem}

\medskip

Although the original motivation of this note was to write down this recurrence formula, it led us to studying the variable $S_k$ in itself, leading to the following results whose proofs occupy most of the paper.
\begin{theorem}\label{thm:expectSk}
	One has, in probability (and in fact, in expectation, and in $L^2$), when $k\longrightarrow \infty$,
	$$\frac{S_k}{\sqrt{k}\log k} \longrightarrow \frac{1}{\sqrt{2\pi}}.$$
\end{theorem}

\begin{theorem}\label{thm:deviationSk}
	There are positive constants $c,c'$  such that for $k\geq 1$,
\begin{align}\label{eq:lowerdev}
	\mathbf{P} (S_k<c\sqrt{k}) \leq \exp(-c' \sqrt{k}).
\end{align}
\end{theorem}
\begin{corollary}\label{cor:asymptMoments}
	Let $p\in(0,1)$. There are constants $c,c',c''$ such that 
	\begin{align}\label{eq:asympt}
		\exp(-c\sqrt{k}\log k)\leq \mathbf{E}\Lambda_p^k \leq \exp(-c'\sqrt{k}),
	\end{align}
	and
	\begin{align}\label{eq:asympt2}
		\mathbf{P}(\Lambda_p <\epsilon) = \mathbf{P}(\Lambda_{1-p} >1-\epsilon)\leq \exp(-c'' \epsilon^{-1}).
	\end{align}
\end{corollary}


\begin{figure}
	\centering
	\begin{minipage}{0.4\linewidth}
	\includegraphics[width=4cm]{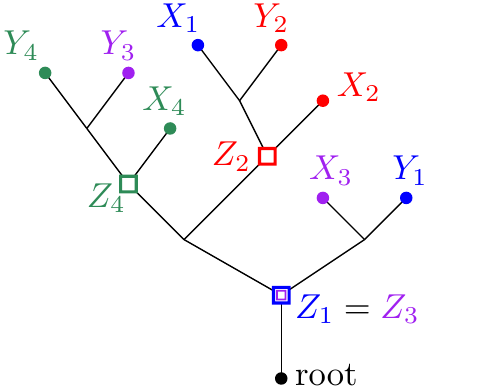}
	\end{minipage}
	\begin{minipage}{0.4\linewidth}
	\includegraphics[width=4.5cm]{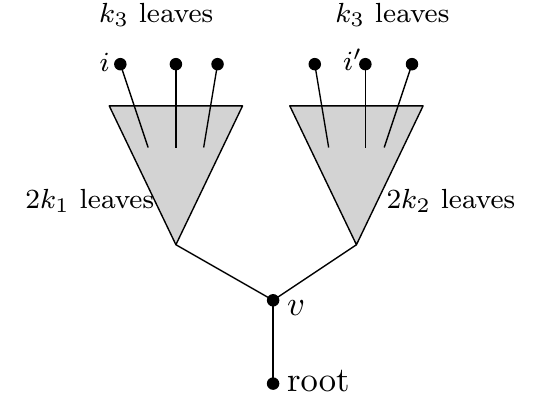}
	\end{minipage}
\caption{Left: Realization of the random variable $S_k$, for $k=4$. On this example $S_k=3$ since the four pairs of vertices $(X_i,Y_i)$ for $i\in [4]$ share three last common ancestors, namely $Z_1=Z_3$, $Z_2$, and $Z_4$. Right: Origin of the recurrence relation of Theorem~\ref{thm:main}. The $k_3$ highlighted pairs of leaves are pairs of the form $(i,i')$ which are separated by the vertex $v$. The vertex $v$ contributes to $S_k$ if and only if $k_3 \neq 0$.}
	\label{fig:rec}
\end{figure}

\section{Comments, questions, and plan} 


Equation~\eqref{eq:mainrec} solves the question of the practical calculation of the moments of~$\Lambda_{1/2}$ raised by the authors of~\cite{bbfgp,bbfgmp}, which was the starting point of this work. For example, we get $a_0(\tfrac{1}{2}), \dots, a_6(\tfrac{1}{2}) =$ {\small $1$, $1/2$, $17/60$, $7/40$, $6361/55440$, $1741/22176$, $154917299/2793510720$}, the first five values being in accordance with~\cite{bbfgp} (the random variable $\Lambda_{1/2}$ already appears prior to~\cite{bbfgmp} in the paper~\cite{bbfgp} in the context of separable permutations, and this is where the five first moments can be found).

The equation~\eqref{eq:mainrec} has some similarity to the equations written by Janson~\cite{janson} for the (limiting) moments of the Wiener index of large random trees, and in fact our method is directly inspired from that paper. The recurrence of~\cite{janson} is also strongly related (albeit for unexplained reasons, see e.g.~\cite{Chapuy:HDR}) to the quadratic recurrence formula that enables one to compute the map asymptotic constant $t_g$ driving the enumeration of maps of genus $g$, see~\cite{BGR} (see also~\cite{BDGRW} and references therein).
It would be interesting to be able to determine the asymptotics of our sequence $a_k(p)$ directly from the relation~\eqref{eq:mainrec}, either by using some ideas of these papers or by developing new ones. This is not how we proceed here, and Corollary~\ref{cor:asymptMoments} is a consequence of our independent study of the variable $S_k$. We hope that other methods could close the gap between the two sides of~\eqref{eq:asympt}, and lead to a true equivalent of the moments.

In fact, most of this paper is devoted to the  proof of convergence of $S_k/(\sqrt{k}\log k)$. We use the second moment method, which requires to estimate the asymptotics of the explicit sums expressing the first two moments of $S_k$.
In order to do this we rely on approximation theorems for variants of Riemann sums, giving results of the form
$$
\sum_{1\leq i,j\leq K} f(\frac{ij}{K}) \sim K\log K \int f, \ \ K\longrightarrow \infty,
$$
(for suitable functions $f$) which we have not been able to find in the literature, see  Theorem~\ref{thm:all-in-1dim} and Lemma~\ref{lemma:divisorSum} in Section~\ref{subsec:approx}. The extra logarithmic factor $\log K$ in these statements, compared to the usual approximation of Riemann sums by integrals, reflects the averaged behaviour of the divisor function
$$
n \longmapsto |\{(i,j),ij=n\}|,
$$
which is very irregular but whose average on large intervals is logarithmic, a result that goes back to Dirichlet (see e.g.~\cite[18.2]{HardyWright}). Hence in some sense the log-factor in Theorem~\ref{thm:expectSk} has an arithmetic origin, at least with our proofs.
It would be interesting to prove the convergence in law of $S_k/(\sqrt{k}\log k)$ with another method, in which the large scale behaviour would occur as the limit of some related process -- this remains to be done.

More generally, our paper poses the question of studying further the variable $S_k$. In particular, our methods leave open the question of giving good (and two-sided) large deviation estimates on $S_k$ (the one-sided estimates of Theorem~\ref{thm:deviationSk} are obtained by a connection with random interval graphs which is convenient but does not capture at all the whole complexity of the problem, missing in particular the logarithmic factor). Finally, as mentionned above, closing the gap in~\eqref{eq:asympt} and obtaining precise asymptotics for the moments  remains to be done.

\bigskip

The rest of the paper is devoted to the proofs. The three sections~\ref{sec:skeletons},~\ref{sec:lower},~\ref{sec:convergence} are independent.
In Section~\ref{sec:skeletons}, we prove the recurrence formula (Theorem~\ref{thm:main}), by observing that the underlying combinatorial objects have a simple decomposition (Figure~\ref{fig:rec}-Right says it all).
In Section~\ref{sec:lower}, we study lower deviations of $S_k$ via a connection with random interval graphs and Poisson point processes, to prove Theorem~\ref{thm:deviationSk}. 
Finally in Section~\ref{sec:convergence} we prove the convergence of $S_k/(\sqrt{k} \log k)$. We use the second moment method, which requires to estimate the first (Section~\ref{sec:firstmoment}) and second (Section~\ref{sec:secondmoment}) moments of $S_k$. In order to do this we rely on approximation theorems for variants of Riemann sums presented in Section~\ref{subsec:approx}.
Finally, the very short Section~\ref{sec:remaining} gives the proof of Corollary~\ref{cor:asymptMoments}.


\section{Skeletons and proof of Theorem~\ref{thm:main}}
\label{sec:skeletons}

Let as before $X_1,Y_1,\dots,X_k,Y_k$ be i.i.d. points chosen according to the Lebesgue measure on the CRT $\mathcal{T}_\mathbf{e}$, and let $Z_i=X_i \wedge Y_i$  for $i\in [k]$. 
Let $T_k$ be the combinatorial skeleton of $X_1,Y_1,\dots,X_k,Y_k$ in $\mathcal{T}_\mathbf{e}$. Almost surely, $T_k$ belongs to the set $\mathbb{T}_k$ formed by planted\footnote{planted=with unary root} binary trees with $2k$ (numbered) leaves. Moreover, the law of $T_k$ is uniform on $\mathbb{T}_k$. 

Now, we have $|\mathbb{T}_k|=(2k)!\Cat{2k-1}$ where  $\Cat{x}=\frac{1}{x+1}{2x\choose x}$ is the $x$-th Catalan number. 
Given a tree in $\mathbb{T}_k$, with leaves numbered $1,1', 2,2', \dots, k,k'$, consider the set $\mathrm{S}(T)$ of last common ancestors
$$
\mathrm{S}(T)=\{1\wedge 1', 2\wedge 2', \dots, k\wedge k'\}.
$$
In previous notation, we have $S_k=|S(T_k)|$. It follows that
$
\mathbf{E}p^{S_k} = \dfrac{b_k(p)}{(2k)!\Cat{2k-1}},
$
where $b_k(p)=\displaystyle \sum_{T\in \mathbb{T}_k} p^{|S(T)|}$. 

We now give a recursion on $b_k(p)$ via a simple root decomposition, see Figure~\ref{fig:rec}-Right. The paper~\cite{janson} was an inspiration. Let $T\in \mathbb{T}_k$. Consider the child $v$ of the root, and let $S_1,S_2$ be its left and right subtrees. For $r\in [2]$, let $k_r$ be the number of $i\in [k]$ such that both $i$ and $i'$ appear in $S_r$. Let $k_3=k-k_1-k_2$, which is the number of pairs $i,i'$ such that $v=i\wedge i'$.
All configurations corresponding to a given triple $(k_1,k_2,k_3)$ can be constructed as follows:
\begin{itemize}[itemsep=0pt, topsep=0pt,parsep=0pt, leftmargin=24pt]
	\item Choose which labels $i\in [k]$ will correspond to leaves counted by $k_1,k_2,k_3$. For the ones counted by $k_3$, choose which element of the pair will go to $S_1$ or $S_2$ ($\frac{k!}{k_1!k_2!k_3!}\times 2^{k_3}$ choices).
\item Choose a left subtree $G_1$ in $\mathbb{T}_{k_1}$ and a right subtree $G_2$ in $\mathbb{T}_{k_2}$ (and relabel leaves according to the previous choice).
\item Attach successively $k_3$ labelled leaves to $G_1$, and $k_3$ labelled leaves to $G_2$, to construct the subtrees $S_1$ and $S_2$ of $v$ in $T$.
\end{itemize}
To attach the leaves in the last step, we use Rémy's construction~\cite{Remy}. Namely, since $T_{k_1}$ has $2k_1$ leaves, it has $4k_1-1$ edges (we assume $k_1>0$ for the moment). The number of ways to attach a leaf is $2(4k_1-1)$ since a leaf can be attached to either side of each edge. The number of edges increases by $2$ at each attachment. Iterating, the number of ways to attach $k_3$ numbered leaves to a tree in $\mathbb{T}_{k_1}$ is
$$
F(k_1,k_3) := \prod_{i=0}^{k_3-1}2(4k_1-1+2i) = 
\frac{(4k_1+2k_3-3)!!2^{2k_1+k_3-1}}{(2k_1)!\Cat{2k_1-1}}.
$$
The latter formula is also valid when $k_1$ is equal to $0$, with the convention $\Cat{-1}=1$.
Putting everything together, we obtain the recursion:
\begin{align}\label{eq:firstrec}
	b_k(p) = \sum_{k_1+k_2+k_3=k, k_i\geq 0\atop{k_1+k_3>0 \mbox{\tiny ~and~} k_2+k_3>0}}
	\frac{2^{k_3}k!}{k_1!k_2!k_3!}b_{k_1}(p)b_{k_2}(p) F(k_1,k_3) F(k_2,k_3) p^{\mathbf{1}_{k_3\neq0}}.
\end{align}
Here we have used that $|\mathrm{S}(T)|=|\mathrm{S}(G_1)|+|\mathrm{S}(G_2)|+\mathbf{1}_{k_3\neq0}$, since the vertex $v$ separates some pair $i,i'$ if and only if $k_3\neq0$.
Note also that numbers $k_1,k_2,k_3$ can be zero, but we need $k_1+k_3$ and $k_2+k_3$ to be nonzero since the subtrees $S_1,S_2$ are nonempty by definition.

The recurrence~\eqref{eq:mainrec} is equivalent to~\eqref{eq:firstrec}  using that $a_k(p)=\frac{b_k(p)}{(2k)!\Cat{2k-1}}$ for all $k\geq 0$, so Theorem~\ref{thm:main} is proved.

\section{Lower deviations of $S_k$ via random interval graphs}
\label{sec:lower}

In this section we prove Theorem~\ref{thm:deviationSk} via a connection with random interval graphs.

Let $A_1,B_1,\dots,A_k,B_k$ be i.i.d. uniform variables on $[0,1]$. We can realize the variables $X_i,Y_i$ considered before by taking $X_i=\pi_\mathbf{e}(A_i), Y_i=\pi_\mathbf{e}(B_i)$, $i\in [k]$. Let $I_i$ be the interval $[\min(A_i,B_i),\max(A_i,B_i)]$ for $i\in [k]$.
Consider the graph $G_k$ on $[k]$ with an edge between $i$ and $j$ if and only if $i\neq j$ and the intervals $I_i$ and $I_j$ intersect. $G_k$ is called a \emph{random interval graph}, see~\cite{intervalGraphs}. Note that the law of $G_k$ would be the same if we replaced the uniform distribution on $[0,1]$ by any atom-free distribution on $\mathbb{R}$. Let $\alpha(G_k)$ be the size of the  largest independent set in the graph $G_k$, then we clearly have, a.s.,
\begin{align}\label{eq:GkSk}
\alpha(G_k)\leq S_k
\end{align}
since the minima of $\mathbf{e}$ on mutually disjoint intervals $I_i$ are a.s. all distinct.
It is proved in\cite{intervalGraphs} that $\alpha(G_k)\sim \frac{2}{\sqrt{\pi}}\sqrt{k}$, in probability. The proof of~\cite{intervalGraphs} is easily adapted to show the following
\begin{proposition}\label{prop:devInterval}
	For any $\epsilon>0$, we have $\mathbf{P}\left(\alpha(G_k)<\left(\dfrac{2}{\sqrt{\pi}}-\epsilon\right)\sqrt{k}\right)\leq\exp(-c\sqrt{k})$, for some $c=c(\epsilon)>0$.
\end{proposition}
Note that although the precise value of the constant $\frac{2}{\sqrt{\pi}}$ is sharp for $\alpha(G_k)$, many other factors are lost a priori in the inequality~\eqref{eq:GkSk} (even possibly logarithmic factors). 
\begin{proof}
	We follow the proof in~\cite{intervalGraphs}. Consider a Poisson point process of intensity $1$ on $(\mathbb{R}_{\geq 0})^2$. Let $(U_1,V_1),(U_2,V_2), \dots$ be a subsequence of points from that process chosen as follows: $(U_1,V_1)$ is the point that minimizes $\max(U_1,V_1)$, and for and $j\geq 2$, $(U_j,V_j)$ is the point that minimizes $\max(U_j,V_j)$ among points whose coordinates are larger than $\max(U_{j-1},V_{j-1})$.

	Each point $(u,v)$ of the point process defines a real interval $[\min(u,v),\max(u,v)]$. If we let $\tau_k$ be the smallest $R$ such that $[0,R]^2$ contains $k$ points of the process, then the intersection graph of the $k$ corresponding intervals has the same law as $G_k$. Moreover, let $F_R=\max\{p,(U_p,V_p) \in [0,R]^2\}$, then the points $(U_1,V_1),\dots,(U_{F_{\tau_k}},V_{F_{\tau_k}})$ define an independent set of size $F_{\tau_k}$ in this graph (and in fact, a largest independent set, see~\cite{intervalGraphs}).

	We thus only have to estimate the probability that $F_{\tau_k}$ is small.
	Choose $R=(1-\epsilon)\sqrt{k}$ with $\epsilon>0$. The number of points in $[0,R]^2$ is a Poisson variable of parameter $(1-\epsilon)^2k$, 
	therefore the probability that it is more than $k$ is at most $e^{-c'k}$ for some $c'>0$, by the Chernoff inequality. Therefore, $\tau_k>R$ with probability at least $1-e^{-c'k}$. Note that if this event holds, we have $F_{\tau_k}\geq F_R$. 

	Now, as observed in~\cite{intervalGraphs}, the differences $\max(U_{j+1},V_{j+1})-\max(U_j,V_j)$ form a sequence of i.i.d real random variables with mean $\frac{\sqrt{\pi}}{2}$ and exponentially bounded moments. Hence, for $\epsilon$ small enough, using the Chernoff inequality again, the probability that the sum of the first $\left\lfloor\frac{2}{\sqrt{\pi}}(1-\epsilon) R\right\rfloor$ of them is more than $R$ is at most $e^{-c''R}$. Note that this is the same as the probability that $F_{R} \leq \frac{2}{\sqrt{\pi}}(1-\epsilon) R$.

	Hence the probability that $F_{\tau_k}$ is at most $\left\lfloor\frac{2}{\sqrt{\pi}}(1-\epsilon) R\right\rfloor$ is at most $e^{-c'k}+e^{-c''R}$, and since $R=(1-\epsilon)\sqrt{k}$, we are done (up to adapting the value of $\epsilon$).
\end{proof}

\begin{proof}[Proof of Theorem~\ref{thm:deviationSk}]
Theorem~\ref{thm:deviationSk} directly follows from the last proposition and~\eqref{eq:GkSk}.
\end{proof}

\section{Convergence of $S_k/(\sqrt{k} \log k)$}
\label{sec:convergence}

We now turn to the proof of Theorem~\eqref{thm:expectSk}. We use the second moment method, requiring to estimate the first two moments of $S_k$. Theorem~\eqref{thm:expectSk} follows from Propositions~\ref{prop:expectSk} and~\ref{prop:secondSk} below.

We start in Section~\ref{subsec:approx} by stating some lemmas (not directly related to trees and independent from the rest of the paper) that we need to approximate sums by integrals, and that we haven't been able to find in the literature.

\subsection{Sum/integral approximation statements }
\label{subsec:approx}

We give first a statement on compact subsets of $(0,\infty)$, and then a "dominated"  version that we will need on the full interval.

\begin{lemma}\label{lemma:divisorSum}
	Let $g:\mathbf{R}_{>0}\rightarrow \mathbf{R}_{>0}$ be continuous, and let $0<\epsilon <\Lambda<\infty$. 
	Let $A=A(K)$ and $B=B(K)$ with $1\leq A<<B\leq K$ and moreover $\log A = o(\log K)$ and $\log B \sim \log K$.

	Then one has
	$\displaystyle\sum_{A\leq p,q \leq B, \atop \epsilon K \leq pq \leq \Lambda K} g\left(\frac{pq}{K}\right)
	\sim  K \log K 	\int_{[\epsilon,\Lambda]} g.
$
\end{lemma}

\begin{theorem}\label{thm:all-in-1dim}
	Let $f:(\mathbb{Z}_{>0})^3 \longrightarrow \mathbb{R}$.
	 Assume that
	there exist functions $A=A(K),B=B(K)$ such that $1\leq A<<B\leq K$ and moreover $\log B \sim \log K$ and $\log A = o(\log K)$, such that
\begin{itemize}
	\item[(i)]
		$\displaystyle \sum_{{1\leq p,q\leq K}\atop \max(p,q)>B} f(p,q,K) = o( K \log K)$
	\item[(ii)] 
There is a constant $C>0$ such that 
for $p,q\leq B$, one has
		$$
		|f(p,q,K)| \leq C \min\left(\left(\frac{pq}{K}\right)^{-1/2},\left( \frac{pq}{K}\right)^{-3/2}\right),
		$$
	\item[(iii)]
		There exist an integrable function $g:\mathbb{R}_{>0}\rightarrow \mathbb{R}_{>0}$ such that
		$$f(p,q,K)\sim g(\frac{pq}{K}),$$
		uniformly for $A\leq p,q \leq B$ and $\frac{pq}{K}$ in any compact subset of $\mathbf{R}_{>0}$.
\end{itemize}
Then one has, when $K$ goes to infinity,
$$
	\sum_{1\leq p,q\leq K} f(p,q,K) \sim K \log K \int_{0}^\infty g.
$$

\end{theorem}

\begin{proof}[Proof of Theorem~\ref{thm:all-in-1dim}]
	Fix a compact $[\epsilon,\Lambda]\subset \mathbb{R}_{>0}$.

	Write $H(I)=\sum\limits_{{1\leq p,q\leq K}\atop{(p,q)\in I}} f(p,q,K)$.
	We first estimate $H(\{p,q\leq B, \frac{pq}{K}>\Lambda\})$. We have, by (ii),
	\begin{align*}
H(\{p,q\leq B, \frac{pq}{K}>\Lambda\}) 
		&\leq C \sum_{{A\leq p,q \leq \Lambda K} \atop pq\geq \Lambda K} (\frac{pq}{K})^{-3/2}\\
		&\leq CK^{3/2} \sum_{A\leq p \leq \Lambda K} p^{-3/2} \sum_{q\geq \frac{\Lambda K}{p}} q^{-3/2} \\
		&\leq  K^{3/2} (\Lambda K)^{-1/2}  \sum_{A\leq p \leq \Lambda K} O(p^{-1}) \\
		&\leq  O\big(\Lambda^{-1/2} K \log \Lambda K \big). \\
\end{align*}

	We now estimate $H(\{p,q\leq B, \frac{pq}{K}<\epsilon\})$. We have, by (ii) again,

\begin{align*}
H(\{p,q\leq B, \frac{pq}{K}<\epsilon\})
	&\leq  C K^{1/2} \sum_{{p\leq B}} p^{-1/2} \sum_{q\leq \lfloor \epsilon K /p\rfloor} q^{-1/2} \\
	&\leq K^{1/2} \sum_{{p\leq B}} p^{-1/2} O\left(1+\sqrt{\frac{\epsilon K}{p}}\right)
	\\
	&\leq O\left(B^{1/2}K^{1/2} + B^{1/2}K \sqrt{\epsilon} \sum_{{p\leq B}} p^{-1}\right)
	\\&= O\left(K (1+\sqrt{\epsilon} \log K)\right),
	\end{align*}
	where we used $B\leq K$.

	We now estimate $H(\{p,q\leq B, \frac{pq}{K}<\Lambda, p< A\})$. Using (ii) again, 
	\begin{align*}
H(\{p,q \leq B, \frac{pq}{K}<\Lambda, p< A\}) 
		&\leq CK^{1/2} 	\sum_{{p< A, q\leq B}\atop pq\leq \Lambda K } 
		(pq)^{-1/2} \\
		&\leq C K^{1/2}\sum_{{p< A} } 
		p^{-1/2} \sum_{q\leq \Lambda K /p}  q^{-1/2} )
		\\
		&= \Lambda^{1/2} K  O(\sum_{p< A} p^{-1}) 
		\\
		&= \Lambda^{1/2} K   O(\log A) =  \Lambda^{1/2} K   o(\log K). 
	\end{align*}
By symmetry of hypotheses the same estimate holds for $H(\{p,q\leq B, \frac{pq}{K}<\Lambda, q< A\})$.

	Putting all estimates together, including $(i)$, the quantity to estimate is equal to
\begin{align*}
	H([K]^2) = K \left( O(1)+ O(\sqrt{\epsilon} \log K) + O(\Lambda^{-1/2} \log \Lambda K )+ (1+\Lambda^{1/2}) o(\log K)\right)
	+H(\{A\leq p,q \leq B,\epsilon<\frac{pq}{K}<\Lambda\})
\end{align*}
Now, from (iii) and Lemma~\ref{lemma:divisorSum}, we have 
$$
	H(\{A\leq p,q \leq B,\epsilon<\frac{pq}{K}<\Lambda\})\sim  K \log K \int_{\epsilon}^\Lambda g.
$$
Taking the limit $K\rightarrow \infty$, it follows that any accumulation point of the sequence $H([K]^2)/(K\log K)$ is equal to 
$$
\left(\int_{\epsilon}^\Lambda g \right)+ O(\sqrt{\epsilon}+\Lambda^{-1/2}).
$$
Now we let $\epsilon$ go to zero and $\Lambda$ go to infinity, and we deduce that $H([K]^2)/(K\log K)$ has a unique accumulation point, equal to $\int_{0}^\infty g$, which is what we wanted to prove.
\end{proof}

\begin{proof}[Proof of Lemma~\ref{lemma:divisorSum}]
	The sum $S$ to evaluate is equal to $\sum_{\epsilon K \leq j \leq \Lambda K} g(\frac{j}{K}) \sigma(j)$ where $\sigma(j)=|\{(p,q),A\leq p,q\leq B, pq=j\}|$ is a restricted variant of the divisor function. The idea of the proof is as follows:  the divisor function is irregular but it behaves as a logarithm on average, and the continuity of $g$ is enough to guarantee that this average behaviour is the effect we observe in our case. We can thus pull the logarithm out, and estimate the remaining sum by an integral in a standard way. Let us proceed with the proof, in particular we have to be careful on the scale of the restrictions $A,B$ imposed to $p$ and $q$. 

	Let $\ell \geq 1$ be a fixed integer, and let us split the interval $[\epsilon K, \Lambda K)$ into $\ell$ intervals $I_1,\dots,I_\ell$ of width $\delta K$ with $\delta= (\Lambda-\epsilon)/\ell$, with $I_i=[a_{i-1},a_{i})$ and $a_i=(\epsilon+i\delta)K$.

	For $i\in [\ell]$, we have 
	\begin{align*}
		\sum_{j\in I_i} \sigma(j) = \sum_{p=A}^B \sum_{q=\max(A,a_{i-1}/p)}^{\min(B, a_i/p)} 1 
		&= \sum_{p=A}^B \lfloor \min(B, a_i/p)\rfloor - \lfloor \max(A,a_{i-1}/p) \rfloor\\
		& = O(B-A) + \sum_{p=A}^B \min(B, a_i/p) - \max(A,a_{i-1}/p).
	\end{align*}
	Note that $a_i/B < a_{i-1}/A$ for $K$ large enough (uniformly).
	We split the last sum in parts. The contribution of $p<a_i/B$ is at most (taking a worse value for the minimum)
	$$
	\sum_{p\leq a_i/B} a_i/p - a_{i-1}/p = (a_i-a_{i-1}) O(\log K/B) = (a_i-a_{i-1}) o(\log K),
	$$
	where we used $\log K \sim \log B$ so $\log K -\log B =o(\log K)$. The contribution of $a_{i}/B<p<a_{i-1}/A$ is  equal to
	$$ \sum_{p=a_i/B}^{a_{i-1}/A} a_i/p -a_{i-1}/p =  (1+o(1))(a_i-a_{i-1}) \log \frac{Ba_{i-1}}{Aa_{i}}
 =  (1+o(1))(a_i-a_{i-1}) \log K,$$
	where we used that $a_{i}/a_{i-1} = \Theta(1)$ and that $\log B -\log A \sim \log K$.
	Finally,
	The contribution of $p>a_{i-1}/A$ is at most (taking a worse value for the maximum)
	$$
	\sum_{p\geq a_{i-1}/A}^B a_i/p - a_{i-1}/p = (a_i-a_{i-1}) O(\log AB/K) = (a_i-a_{i-1}) o(\log K).
	$$
We have thus proved that 
	\begin{align}\label{eq:averageDivisor}
	\sum_{j\in I_i} \sigma(j)  = O(B-A) + (a_i-a_{i-1}) o(\log K) + (1+o(1))(a_i-a_{i-1}) \log K
	=(1+o(1))\delta K \log K +o(K \log K),
	\end{align}
where the big and little-o are independent of $\delta$.
	Now, since $g$ is absolutely continuous (we are on a compact set), we have, for $j\in I_i$, 
	$g(\frac{j}{K}) = g(\frac{a_{i-1}}{K}) + O(r_\delta)$, uniformly on $i$ and $j$, with $r_\delta\rightarrow 0$ when $\delta\rightarrow 0$.

	Putting all estimates together we obtain:
	\begin{align*}
		S &=
		\sum_{i=1}^\ell \Big(g(\frac{a_{i-1}}{K}) +O(r_\delta)\Big) ((1+o(1))\delta K \log K + o(K \log K)) 
\end{align*}
	We can finally $\ell$ go to infinity (i.e. $\delta$ go to zero) slowly enough, so that $\sum_{i=1}^\ell \delta g(\frac{a_{i-1}}{K}) $ converges to $\int_{[\epsilon,\Lambda]} g$, and we are done.
\end{proof}

\subsection{First moment}
\label{sec:firstmoment}

This section is dedicated to the proof of the following proposition.
\begin{proposition}\label{prop:expectSk}
One has, when $k\rightarrow\infty$,
	$\displaystyle \mathbf{E}S_k\sim \frac{1}{\sqrt{2\pi}}\sqrt{k}\log k.$
\end{proposition}

	First, we have $\mathbf{E}S_k=(K-1)\mathbf{P} (v\in \mathcal{S}_k)$ where the probability is on $T_k$ as before and on a vertex $v$ chosen uniformly at random among the $K-1$ internal vertices of $T_k$, with $K=2k$.
Therefore we want to prove that
$$\mathbf{P} (v\in \mathcal{S}_k)\sim \frac{1}{2\sqrt{\pi}}K^{-1/2}\log K.
$$

We let $T_v$ be the subtree of $v$ in $T_k$, and $T_v^L$, $T_v^R$ its left and right subtrees.
	The vertex $v$ belongs to $\mathcal{S}_k$ if and only if there is $i\in [k]$ such that $i\in T_v^L, i'\in T_v^R$, or the converse. Therefore we have (we use $|\cdot|$ for the number of leaves)
	\begin{align}\label{eq:Ev}
		\mathbf{P} (v\in \mathcal{S}_k)=\mathbf{E} C(|T_v^R|, |T_v^L|, K),
	\end{align}
where for $p+q\leq K$, $C(p,q,K)$ is the probability that for a uniform random matching on $[K]$, there exist $i\in [p], j\in p+[q]$ which are matched together.
Note that we can write this explicitly:
	\begin{align}\label{eq:EvExplicit}
		\mathbf{E} C(|T_v^R|, |T_v^L|, K)=\sum_{p+q+r=K+1\atop p,q,r \geq 1}
		P(p,q,K)
		C(p,q,K),
	\end{align}
	where 
	\begin{align}
		P(p,q,K)=\mathbf{P}(T_v^R=p,T_v^L=q)&=
		\frac{1}{(K-1)\Cat{K-1}}\Cat{p-1} \Cat{q-1} \Cat{K-p-q}(K+1-p-q) \label{eq:Pexact}\\
\label{eq:CpqKexact}
		C(p,q,K) &= \sum_{\ell\geq 1} (-1)^{\ell+1} {p \choose \ell} {q \choose \ell} \ell! \frac{(K-2\ell-1)!!}{(K-1)!!}
		= \sum_{\ell\geq 1} \frac{(-1)^{\ell+1} (p)_\ell (q)_\ell}{(K)^{(2)}_\ell \ell!},
	\end{align}
where we use the notation $(m)_\ell=m(m-1)\dots(m-\ell+1)$ and $(m)^{(2)}_\ell=(m-1)(m-3)\dots(m-2\ell+1)$. 
	In formulas~\eqref{eq:EvExplicit}-\eqref{eq:Pexact}, $r=K+1-p-q$ represents the number of leaves of the tree that remains after deleting all subtrees of $v$: this tree is one of the $\Cat{r-1}\times r$ planted binary trees with one marked leaf and $r$ leaves in total, while the two subtrees of $v$ are among the $\Cat{p-1}\Cat{q-1}$ pairs of trees of the given size: see Figure~\ref{fig:proba1}-Left.
	Formula~\eqref{eq:CpqKexact} is a direct consequence of the inclusion-exclusion formula (the $\ell$-th summand, without sign, counts matchings on $[K]$ with $\ell$ distinguished matched pairs $(i,j)$ such that $i\in[p], j\in p+[q]$, divided by the total number $(K-1)!!$ of matchings on $[K]$).

\begin{figure}
	\centering
	\begin{minipage}{\linewidth}\centering
	\includegraphics[height=30mm]{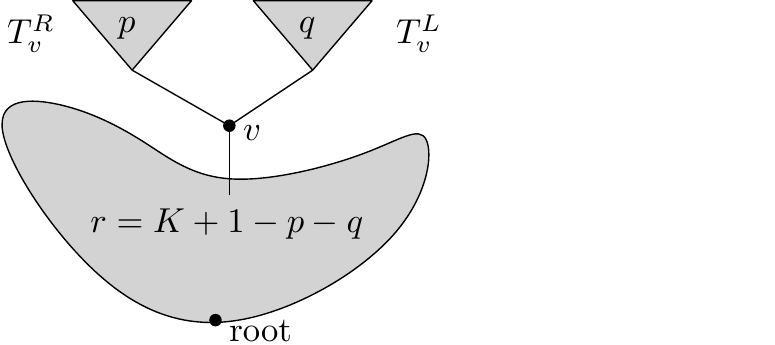}
	\hspace{1cm}
	\includegraphics[height=30mm]{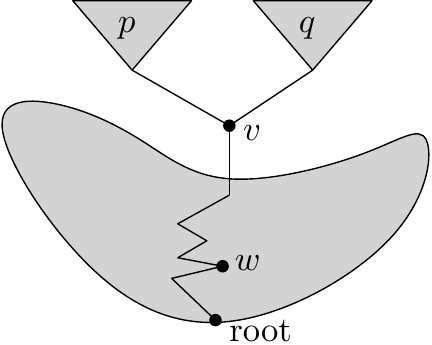}
	\end{minipage}
	\caption{Left: Origin of the Catalan numbers in~\eqref{eq:Pexact}. Labels $p,q,r$ indicate the size of each highlighted tree (number of leaves, including $v$ for the one containing the root). The whole tree has $K=2k$ leaves in total, and $K-1$ internal vertices.
	Right: the variant used in~\eqref{eq:EvExplicitTilde} in Section~\ref{sec:secondmoment}.
	}
	\label{fig:proba1}
\end{figure}

We will apply Theorem~\ref{thm:all-in-1dim} above, for the function
$$
f(p,q,K):= K^{3/2} P(p,q,K) C(p,q,K)
$$
with $A=1$ and $B=K/\log K$. We have to check the three hypotheses of that theorem.

\smallskip
{\noindent \bf (i) Contribution of $|T_v|>B$.}
Recall that the Catalan numbers satisfy the asymptotics 
	\begin{align}\label{eq:largeCat}
		\Cat{m} \sim \frac{1}{\sqrt{\pi}}m^{-3/2}4^m \ \ , \ \ m\longrightarrow \infty.
	\end{align}
It follows (classically) that, for $B=K/\log K$,  
\begin{align}
	\mathbf{P}(|T_v|>B) \leq \sum_{p> B} \frac{\Cat{p}\Cat{K+1-p}(K+1-p)}{(K-1)\Cat{K-1}} &= \sum_{p>B} O(p^{-3/2}(1-\tfrac{p-1}{K})^{-1/2})\label{eq:classicalBoundLargeSubtree}\\
	&=\sum_{B<p<K/2} O(p^{-3/2}(1-\tfrac{p-1}{K})^{-1/2}) + \sum_{K/2\leq B \leq K} O(p^{-3/2}(1-\tfrac{p-1}{K})^{-1/2}) \nonumber
	\\
	&= O(B^{-1/2}) + O(K^{-1/2}) = O(B^{-1/2}) = o(K^{-1/2}\log K). \nonumber
\end{align}
Since $\max(|T_v^L|,|T_v^R|)>B$ implies $|T_v|>B$, the same upper bounds holds for $\mathbf{P}(\max(|T_v^L|,|T_v^R|)>B)$, which exactly says that hypothesis (i) of Theorem~\ref{thm:all-in-1dim} is satisfied for our choice of $f$ and $B$.

\smallskip
{\noindent \bf (ii) Domination bound for $f$.}
The first term in the inclusion-exclusion formula~\eqref{eq:CpqKexact} upper bounds the total sum, thus
\begin{align}\label{eq:CpqKbound}
	C(p,q,K)  \leq \frac{pq}{K-1} = O(\frac{pq}{K}).
\end{align}
Moreover note that, 
uniformly over $p,q\leq B$
we have by~\eqref{eq:largeCat}
		\begin{align}\label{eq:PpqKestimate}
			P(p,q,K)&=\frac{1}{4}\times 4^{1-p}C_{p-1}4^{1-q}C_{q-1} 
		\frac{(K+1-p-q)4^{p+q-K}C_{K-p-q}}{(K-1)4^{1-K}C_{K-1}} \\
			&\sim \frac{1}{4}\times 4^{1-p}C_{p-1}4^{1-q}C_{q-1},\label{eq:PpqKestimate2}
		\end{align}
	since $K-p-q\sim K$ uniformly on this set. This quantity is $O((pq)^{-3/2})$ by~\eqref{eq:largeCat} again. Using either the bound~\eqref{eq:CpqKbound} or the bound $C(p,q,K)\leq 1$, we deduce that there is $C>0$ such that for $p,q\leq B$ one has
$$
|f(p,q,K)| \leq C \min\left(\left(\frac{pq}{K}\right)^{-1/2},\left( \frac{pq}{K}\right)^{-3/2}\right),
$$
i.e. Hypothesis (ii) of Theorem~\ref{thm:all-in-1dim} holds.

\smallskip

{\noindent \bf (iii) Convergence on compact sets.}
We have:
\begin{lemma}\label{lemma:boundCpqK}
	Let $p,q\in [K]$. For all $\epsilon,\Lambda>0$  we have, uniformly for $\epsilon\leq \frac{pq}{K}\leq \Lambda$,
	$$
	C(p,q,K) \sim (1-e^{-\frac{pq}{K}})(1+O(\min(p,q)^{-1/3})).
	$$
\end{lemma}
\begin{proof}
	Consider the $\ell$-th summand in~\eqref{eq:CpqKexact}, and assume $\ell\leq R:=\lfloor \min(p,q)^{1/3}\rfloor$.
	We have $(p)_\ell = p^\ell \exp(\sum_{i=0}^{\ell-1} \log (1-\frac{i}{p})) = p^\ell e^{O(\ell^2/p)}= p^\ell e^{ O(R^{-1})}$. Similarly, $(q)_\ell=q^\ell e^{ O(R^{-1})}$ and $(K)^{(2)}_{\ell}=K^\ell e^{ O(R^{-1})}$.
Moreover, in the inclusion-exclusion formula, the absolute value of the $\ell$-th summand upper bounds the error between the $\ell$-th partial sum and the sum. Now we have, for $\ell =R$, 
$$
	\left|(-1)^{\ell+1}	 \frac{(p)_\ell (q)_\ell}{(K)^{(2)}_\ell \ell!} \right| = \frac{p^R q^{R}}{K^RR!} e^{O(R^{-1})}= e^{O(R)}/R!e^{O(R^{-1})}= e^{-\Omega(R \log R)}
$$
	where we have used $\frac{pq}{K}\leq \Lambda=O(1)$, and the Stirling formula . We thus have:
	\begin{align}
		C(p,q,K)=
		e^{-\Omega(R\log R)}+ e^{ O( R^{-1})} \sum_{\ell= 1}^{R} \frac{(-1)^{\ell+1} p^\ell q^\ell}{K^\ell \ell!}.
	\end{align}
The last sum differs from the infinite sum 
	$\sum_{\ell= 1}^{\infty} \frac{(-1)^{\ell+1} p^\ell q^\ell}{K^\ell \ell!}=1-e^{-\frac{pq}{K}}$ by at most the first neglected summand (again), which decreases as $e^{-\Omega(R\log R)}$ as already seen above.
	Therefore
	$$
C(p,q,K)=
	e^{-\Omega(R\log R)}+(1-e^{-\frac{pq}{K}}+e^{-\Omega(R \log R)})e^{ O( R^{-1})},
	$$
	and we are done, observing that $(1-e^{-\frac{pq}{K}})$ is uniformly bounded away from $0$.
\end{proof}

We deduce from the lemma and from~\eqref{eq:PpqKestimate2} together with~\eqref{eq:largeCat} that
$$
f(p,q,K) \sim \frac{1}{4\pi} (1-e^{-\frac{pq}{K}}) (\frac{pq}{K})^{-3/2}
$$
uniformly for $A\leq p,q\leq B$ and $\frac{pq}{K}\in [\epsilon,\Lambda]$ (note that $p,q\leq B$ and $pq\geq \epsilon K$ imply that $p,q$ go to infinity, so we can apply the asymptotics~\eqref{eq:largeCat} to the Catalan numbers $C_{p-1}$ and $C_{q-1}$ in~\eqref{eq:PpqKestimate2}, and for the same reason the error term in Lemma~\ref{lemma:boundCpqK} goes to zero).
It follows that Hypothesis (iii) of Theorem~\ref{thm:all-in-1dim} holds, with 
$g(x)=\frac{1}{4 \pi}x^{-3/2}(1-e^{-x})$.

\smallskip
{\noindent \bf (iv) Conclusion.} All hypotheses of Theorem~\ref{thm:all-in-1dim} hold for our choice of $A,B,f$, and we deduce that
$$
K^{3/2}\mathbf{P}(v\in \mathcal{S}_k) \sim \left(\frac{1}{4\pi}\int_0^\infty x^{-3/2}(1-e^{-x})dx \right) K \log K. 
$$
The result follows since the last integral is equal to $2\sqrt{\pi}$, which  can be checked by integrating by parts:
$$
\int_0^\infty x^{-3/2}(1-e^{-x})dx=0 - \int_0^\infty (-2x^{-1/2}) e^{-x} dx = 2\Gamma(\frac{1}{2})=2 \sqrt{\pi}.
$$

\subsection{Second moment}
\label{sec:secondmoment}

In this section we prove the following proposition.
\begin{proposition}\label{prop:secondSk}
One has, when $k\rightarrow\infty$,
	$\mathbf{E}(S_k)^2 \sim \left(\frac{1}{\sqrt{2\pi}}\sqrt{k}\log k\right)^2.$
\end{proposition}
The proof is similar in spirit to the one of the first moment.
First, we have
$\mathbf{E}(S_k)^2=(K-1)^2\mathbf{P} (v,w\in \mathcal{S}_k)$ where we now take two uniform independent random internal vertices $v,w$ of $T_k$.
We thus want to show that
\begin{align}\label{eq:secondEq}
	\mathbf{P} (v,w\in \mathcal{S}_k)\sim \frac{1}{4\pi}K^{-1}\log^2 K .
\end{align}


\bigskip

{\noindent \bf $\bullet$ Contribution of $v,w$ having an ancestral relation}.

In order to bound $\mathbf{P} (v,w\in \mathcal{S}_k)$, we first consider the contribution of the event 
$H_v:=\{w \mbox{ is an ancestor of } v \mbox{ in }T_k\}$.
We can upper bound the contribution of this event by forgetting the property that $w \in \mathcal{S}_k$, and only take into account the ancestral relation between $w$ and $v$, and the fact that $v\in \mathcal{S}_k$. 
We can treat this exactly as we treated~\eqref{eq:EvExplicit}, and we get:
	\begin{align}\label{eq:EvExplicitTilde}
\mathbf{P} (v,w\in \mathcal{S}_k, H_v)
		\leq \mathbf{E} [C(|T_v^R|, |T_v^L|, K) \mathbf{1}_{H_v} ]
		=\sum_{p+q+r=K+1\atop p,q,r \geq 1}\tilde{{P}}(p,q,K)   C(p,q,K),
	\end{align}
	where 
	\begin{align}
		\tilde{{P}}(p,q,K)&=
		\frac{1}{(K-1)^2\Cat{K-1}}\Cat{p-1} \Cat{q-1} D(K-p-q) \label{eq:PexactTilde},
	\end{align}
where $D(r)$ is the number of planted plane trees with $r$ leaves, one marked leaf, and a marked vertex on the ancestral line of that leaf: see Figure~\ref{fig:proba1}-right.
Note that, compared to~\eqref{eq:EvExplicit} the only change was to replace $r\Cat{r}$ by $D(r)/(K-1)$.
Now, it is classical, and easy to see, that $D(r) \sim cr^{3/2} \Cat{r}$ for some $c>0$. The whole estimates of the previous sections can then be copied verbatim, with the only difference that an extra factor $\sqrt{K}$ will appear in the calculations in places where we had to estimate a summand of the form $\Cat{K-p-q}/\Cat{K}$ in the sums, and we now have $D(K-p-q)/\mathrm{Cat}(K)$. Once compensated by the extra $(K-1)$ in the denominator, we obtain that 
the three hypotheses of Theorem~\ref{thm:all-in-1dim} hold as in the previous section, with now
$$
f(p,q,K) = K^2 \tilde{{P}}(p,q,K) C(p,q,K).
$$
It follows that
\begin{align}\label{eq:negligible}
	\mathbf{P} (v,w\in \mathcal{S}_k, H_v) = O(K^{-1}\log K). 
\end{align}
Note that this is smaller than the order $K^{-1}\log^2 K$ that we aim to reach in~\eqref{eq:secondEq} (so this contribution will be negligible in the end).
Note also that, by symmetry, the same bound holds if we replace $H_v$ by the symmetric event $H_w$ in which the roles of $v$ and $w$ are reversed.

\bigskip

{\noindent \bf $\bullet$ Contribution of $v,w$ in generic position.}
We now study the contribution of cases where neither $v$ nor $w$ is an ancestor of the other, i.e. we work under the event $H_v^c \cap H_w^c$.

The probability $\mathbf{P} (v,w\in \mathcal{S}_k, H_v^c \cap H_w^c)$ can be written explicitly as
	\begin{align}\label{eq:EvwExplicit1}
	\mathbf{P} (v,w\in \mathcal{S}_k, H_v^c \cap H_w^c)=	\sum_{{p_1+q_1+p_2+q_2+r=K+2\atop p_1,q_1,p_2,q_2,r \geq 1}}
		P(p_1,q_1,p_2,q_2,K)   C(p_1,q_1,p_2,q_2,K),
	\end{align}
	where 
$C(p_1,q_1,p_2,q_2,K)$ is the probability that for a random matching on $[K]$ there is at least one pair matched between $[p_1]$ and $p_1+[q_1]$, and at least one between $p_1+q_1+[p_2]$ and $p_1+q_1+p_2+[q_2]$, and where
\begin{align}
P(p_1,q_1,p_2,q_2,K)
&=\mathbf{P}(|T_v^R|=p_1,|T_v^L|=q_1,|T_w^R|=p_2,|T_w^L|=q_2)\\
&=\frac{\Cat{K+1-\sum p_i - \sum q_i}(K+2-\sum p_i - \sum q_i)_2}{(K-1)^2\Cat{K-1}}\prod_{x= p_1,q_1,p_2,q_2}\Cat{x-1}, \label{eq:Pexact22}
\end{align}
where again Catalan numbers come from the decomposition into subtrees at $v$ and $w$, see Figure~\ref{fig:probaHat}-Left.
Moreover, we have similarly as before, by inclusion-exclusion,
	\begin{align}\label{eq:IE2}
	C(p_1,q_1,p_2,q_2,K)=\sum_{\ell_1, \ell_2 \geq 1} (-1)^{\ell_1+\ell_2} \frac{(K-2\ell_1-2\ell_2-1)!!}{(K-1)!!}\prod_{i=1}^2 {p_i \choose \ell_i} {q_i \choose \ell_i} \ell_i! 
	= \sum_{\ell_1, \ell_2 \geq 1} \frac{(-1)^{\ell_1+\ell_2}\prod_{i=1}^2 (p_i)_{\ell_i} (q_i)_{\ell_i}}{(K)^{(2)}_{\ell_1+\ell_2} \ell_1! \ell_2!},
	\end{align}
since the summand of index $\ell_1,\ell_2$, without sign, counts matchings on $[K]$ having respectively $\ell_1$ and $\ell_2$ distinguished matched pairs between the two pairs of intervals $([p_1],p_1+[q_1])$ and $(p_1+q_1+[p_2],p_1+q_1+p_2+[q_1])$, divided by the total number $(K-1)!!$ of matchings.

We now consider the function
$$
f(p_1,q_1,p_2,q_2,K) := K^{3}P(p_1,q_1,p_2,q_2,K)   C(p_1,q_1,p_2,q_2,K).
$$
We will apply to this function the 2-dimensional analogue of Theorem~\ref{thm:all-in-1dim} (Theorem~\ref{thm:all-in-2dim} below), which requires to check three similar hypotheses, studied below in steps (i), (ii), (iii).

\smallskip

{\noindent \bf (i) Contribution of $|T_v|>B$}.
We first upper bound the contribution of the event $|T_w|>B$ to~\eqref{eq:EvwExplicit1}, for the same value $B=K/\log K$ as in the previous section.
To bound the corresponding probability, we once again forget the condition that $w\in \mathcal{S}_k$, and we only estimate 
$$
\mathbf{P} (v\in \mathcal{S}_k, H_v^c \cap H_w^c, |T_w|>B)
=
\mathbf{E}[C(|T_v^R|,|T_v^L|,K) \mathbf{1}_{H_v^c \cap H_w ^c} \mathbf{1}_{|T_w|>B}],$$ 
which is given by
	\begin{align}\label{eq:EvwExplicit}
		\sum_{{p_1+q_1+s+r=K+2\atop p_1,q_1,s,r \geq 1} \atop s>B}
		\hat{P}(p_1,q_1,s,K)   C(p_1,q_1,K),
	\end{align}
	where
	\begin{align}
		\hat{P}(p_1,q_1,s,K)=
		\frac{\Cat{K+1-p_1-q_1-s}(K+2-p_1-q_1-s)_2}{(K-1)^2\Cat{K-1}}\prod_{x= p_1,q_1,s}\Cat{x-1}, \label{eq:Pexactint}
	\end{align}
	where the Catalan numbers again come from the decomposition into subtrees at $v$ and and $w$, see Figure~\ref{fig:probaHat}-Right.

Now one has, uniformly over $s>B$ and $r\geq 1$, 
$\frac{r \Cat{r-1}\Cat{s-1} }{(r+s)\Cat{r+s-1}} = O(s^{-3/2}(1-\frac{s}{s+r})^{-1/2})$,
by~\eqref{eq:largeCat}, and therefore, uniformly over $y\geq B$,
$$\sum_{{r+s=y}\atop s>B}
\frac{r \Cat{r-1}\Cat{s-1} }{(r+s)\Cat{r+s-1}}=
\sum_{{r+s=y}\atop s>B} s^{-3/2}(1-\frac{s}{s+r})^{-1/2} = O(B^{-1/2}),
$$
by the same calculation as in~\eqref{eq:classicalBoundLargeSubtree}.
We use this observation to substitute the factor $r\Cat{r-1}\Cat{s-1}$ in~\eqref{eq:Pexactint} by $yC_{y-1}$, with $r=K+2-p_1-q_1-s$ and $y=r+s=K+2-p_1-q_1$, and we deduce that
\begin{align}\label{eq:EvwBound1}
\eqref{eq:EvwExplicit}	
	\leq \sum_{{p_1+q_1+y=K+2\atop p_1,q_1,y \geq 1}}
	O(B^{-1/2}) \times
	 C(p_1,q_1,K)\frac{\Cat{K+1-p_1-q_1}(K+2-p_1-q_1)(K-p_1-q_1)}{(K-1)^2\Cat{K-1}}\prod_{x= p_1,q_1}\Cat{x-1},
\end{align}
which is at most $O(B^{-1/2})$ times \eqref{eq:EvExplicit}. Since we have shown in the previous section that the latter is $O(K^{-1/2}\log K)$, it follows that
\begin{align}
	\eqref{eq:EvwExplicit}	 = O(K^{-1} \log^{3/2} K)= o(K^{-1} \log^{2} K).
\end{align}
The same bound holds for the contribution of the stronger event $|T_w^R|>B$, and symmetrically of the events $|T_w^L|>B$, $|T_v^R|>B$, and  $|T_v^L|>B$.

\begin{figure}
	\centering
	\includegraphics[height=30mm]{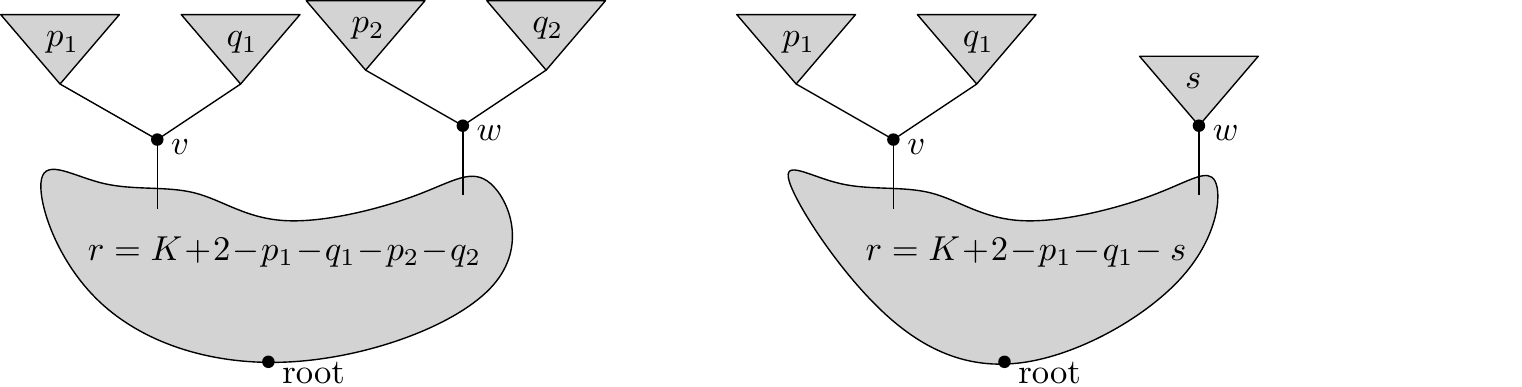}
	\caption{Left: Origin of the Catalan numbers in~\eqref{eq:Pexact22}. 
	Right: the variant used in~\eqref{eq:Pexactint}.
	}
	\label{fig:probaHat}
\end{figure}

\smallskip

{\noindent \bf (ii) Domination bound for $f$}.

We have, uniformly over $p_1,q_1,p_2,q_2\leq B$, that $P(p_1,q_1,p_2,q_2,K)=O((p_1q_1)^{-3/2}(p_2q_2)^{-3/2})$, by~\eqref{eq:largeCat}. Moreover, since the first term in the inclusion-exclusion formula upper bounds the total sum, we also have
$$C(p_1,q_1,p_2,q_2,K) \leq \frac{p_1q_1p_2q_2}{(K-1)(K-3)} = O(\frac{p_1q_1p_2q_2}{K^2}),$$
and moreover $C(p_1,q_1,p_2,q_2,K) \leq C(p_1,q_1,K) =O(\frac{p_1q_1}{K})$, as we have seen in the previous section. 
The quantity $C(p_1,q_1,p_2,q_2,K)$ is also bounded symmetrically by $O(\frac{p_2q_2}{K})$, and of course by $1$.
Putting all these upper bounds in one we have that, uniformly over $p_1,q_1,p_2,q_2\leq B$ we have, for some constant $C>0$,
\begin{align}\label{eq:boundf2}
	|f(p_1,q_1,p_2,q_2,K)| \leq C \prod_{i=1}^{2} \min\left(\left(\frac{p_iq_i}{K}\right)^{-1/2},\left( \frac{p_iq_i}{K}\right)^{-3/2}\right).\end{align}

\smallskip

{\noindent \bf (iii) Convergence on compact sets.}
We have the following analogue of Lemma~\ref{lemma:boundCpqK}.
\begin{lemma}\label{lemma:boundCpqK2}
	Let $p_1,q_1,p_2,q_2\in [K]$. Then we have, uniformly for $\epsilon \leq \frac{p_1q_1}{K}, \frac{p_2q_2}{K}\leq \Lambda$,
	$$
	C(p_1,q_1,p_2,q_2,K) \sim (1-e^{-\frac{p_1q_1}{K}})(1-e^{-\frac{p_2q_2}{K}})(1+O(\min(p_1,q_1,p_2,q_2)^{-1/3}).
	$$
\end{lemma}
\begin{proof}
	Consider~\eqref{eq:IE2}.	The summand is approximated by $\prod_{i=1}^2 \frac{1}{\ell_i!}\left(-\frac{p_iq_i}{K}\right)^{\ell_i}$ exactly as in the proof of Lemma~\ref{lemma:boundCpqK} (details are completely similar and are left to the reader), leading in fine the approximation by the double infinite sum
	$$
	\sum_{\ell_1, \ell_2 \geq 1} \prod_{i=1}^2 \frac{1}{\ell_i!}\left(-\frac{p_iq_i}{K}\right)^{\ell_i}
	= (1-e^{-\frac{p_1q_1}{K}})(1-e^{-\frac{p_2q_2}{K}}). \qedhere
	$$
\end{proof}
If follows that, uniformly over $\epsilon \leq \frac{p_1q_1}{K}, \frac{p_2q_2}{K}\leq \Lambda$ with $p_1,q_1,p_2,q_2\leq B$, we have
$$
f(p_1,q_1,p_2,q_2,K) \sim \frac{1}{16\pi^2} \prod_{i=1}^2 (\frac{p_iq_i}{K})^{-3/2}(1-e^{-\frac{p_iq_i}{K}}),
$$
where we have only taken the asymptotic of Catalan numbers in~\eqref{eq:Pexact22} (as in the previous section, we observe that $p_i,q_i\leq B$ and $\frac{p_iq_i}{K}\geq \epsilon K$ imply that $p_i,q_i$ go to infinity).

\smallskip
{\noindent \bf (iv) Conclusion.}
We will need the two-dimensional analogue of Theorem~\ref{thm:all-in-1dim}.
\begin{theorem}\label{thm:all-in-2dim}
	Let $f:(\mathbb{Z}_{>0})^5 \longrightarrow \mathbb{R}$.
	 Assume that
	there exist functions $A=A(K),B=B(K)$ such that $1\leq A<<B\leq K$ and moreover $\log B \sim \log K$ and $\log A = o(\log K)$, such that
\begin{itemize}
	\item[(i)]
		$\displaystyle \sum_{{1\leq p_1,q_1,p_2,q_2\leq K}\atop \max(p_1,q_1,p_2,q_2)>B} f(p_1,q_1,,p_1,q_2,K) = o( K^2 \log^2 K)$
	\item[(ii)] 
There is a constant $C>0$ such that 
for $p_1,q_1,p_2,q_2\leq B$, one has
		$$
		|f(p_1,q_1,p_2,q_2,K)| \leq C \prod_{i=1}^{2} \min\left(\left(\frac{p_iq_i}{K}\right)^{-1/2},\left( \frac{p_iq_i}{K}\right)^{-3/2}\right),
		$$
	\item[(iii)]
		There exists an integrable function $g:(\mathbb{R}_{>0})^2\rightarrow \mathbb{R}_{>0}$ such that
		$$f(p_1,q_1,p_2,q_2,K)\sim g(\frac{p_1q_1}{K},\frac{p_2q_2}{K}),$$
		uniformly for $A\leq p_1,q_1,p_2,q_2 \leq B$ and $(\frac{p_1q_1}{K},\frac{p_2q_2}{K})$ in any compact subset of $(\mathbf{R}_{>0})^2$.
\end{itemize}
Then one has, when $K$ goes to infinity,
$$
	\sum_{1\leq p_1,q_1,p_2,q_2\leq K} f(p_1,q_1,p_2,q_2,K) \sim K^2 \log^2 K \int_{0}^\infty\int_{0}^\infty g.
$$

\end{theorem}
\begin{proof}
	The proof is similar to the proof of Theorem~\ref{thm:all-in-1dim}.

	The contribution of $\max(p_1,q_1,p_2,q_2)>B$ is taken into account by (i).
	Among remaining terms, the contribution of $(\frac{p_1q_1}{K},\frac{p_2,q_2}{K})$ not in a compact of the form $[\epsilon,\Lambda]^2$ can be upper bounded by (ii), by
	\begin{align*}
		2 &\left(\sum_{{\frac{p_1q_1}{K}\not \in [\epsilon ,\Lambda]} \atop p_1,q_1\leq B}  
\min\left(\left(\frac{p_1q_1}{K}\right)^{-1/2},\left( \frac{p_1q_1}{K}\right)^{-3/2}\right)
	\right)\times 
		\left(	\sum_{p_2,q_2\leq B}  
		\min\left(\left(\frac{p_2q_2}{K}\right)^{-1/2},\left( \frac{p_2q_2}{K}\right)^{-3/2}\right)\right)\\
		&= O(M) \times O(K \log K),
	\end{align*}
with
$
	M= K \left( O(1)+ O(\sqrt{\epsilon} \log K) + O(\Lambda^{-1/2} \log \Lambda K )+ (1+\Lambda^{1/2}) o(\log K)\right).
$
This follows by applying the estimates in the proof of Theorem~\ref{thm:all-in-1dim} to each of the two sums.

	The contribution of $(\frac{p_1q_1}{K},\frac{p_2,q_2}{K})\in[\epsilon,\Lambda]^2$ is analyzed exactly as in the proof of Lemma~\ref{lemma:divisorSum}, by cutting the integration space into small 2-dimensional boxes  and estimating the average of the modified divisor function on each of them. This requires no extra work as in fact all the 1-dimensional estimates in the proof of Lemma~\ref{lemma:divisorSum} can be recycled.
Namely, write again $\sigma(j)=|\{(p,q),A\leq p,q\leq B, pq=j\}|$ and for an integer $\ell \geq 1$, let  $\delta= (\Lambda-\epsilon)/\ell$, and $I_i=[a_{i-1},a_{i})$ and $a_i=(\epsilon+i\delta)K$. For $i,i'\in [\ell]$, we have 
	\begin{align*}
		\sum_{(j,j')\in I_i\times I_{i'}} \sigma(j) \sigma(j') = 
		\sum_{j\in I_i} \sigma(j) \times \sum_{j'\in I_{i'}}\sigma(j')
	\end{align*}
which is 
	$(1+o(1))\delta^2 K^2 \log^2 K +o((K + \delta) K \log^2 K)$ by squaring~\eqref{eq:averageDivisor}.
	One can then use the absolute continuity of $g$ on $[\epsilon,\Lambda]^2$, and take the limit $\delta\rightarrow 0$ as in the proof of Lemma~\ref{lemma:divisorSum}, to conclude that the contribution of $(\frac{p_1q_1}{K},\frac{p_2,q_2}{K})\in[\epsilon,\Lambda]^2$, with $p_1,q_1,p_2,q_2$ is equal to 
	$$ (1+o(1))K^2\log^2K \int_{[\epsilon,\Lambda]^2}g.
	$$

	One then concludes by taking the limit $\epsilon\rightarrow 0$, $\Lambda\rightarrow \infty$ as in the end of the proof of Theorem~\ref{thm:all-in-1dim}. 
\end{proof}

We have already shown, in points (i), (ii), (iii) above,  that our function $f$ satisfies the three hypotheses, with $B=K/\log K$, $A=1$,
and $g(x,y)=\frac{1}{16\pi^2} (xy)^{-3/2} (1-e^{-x})(1-e^{-y})$. 
We conclude that 
$$
K^3 \mathbf{P} (v,w\in \mathcal{S}_k, H_v^c \cap H_w^c)
\sim K^{2} \log ^2 K \times 
\frac{1}{16 \pi^2} \left(\int_0^\infty x^{-3/2}(1-e^{-x}) dx\right)^2
=\frac{1}{4 \pi}K^{2} \log ^2 K,
$$
which is exactly what we need.
This ends the proof of Proposition~\ref{prop:secondSk}.

\section{Remaining proofs}
\label{sec:remaining}

	We conclude by proving Corollary~\ref{cor:asymptMoments}. 
\begin{proof}[Proof of Corollary~\ref{cor:asymptMoments}]
	Write
	$\mathbf{E}p^{S_k} \leq p^{c\sqrt{k}} +  \mathbf{P} (S_k < c\sqrt{k})$ for $c<\frac{2}{\sqrt{\pi}}$, and apply~\eqref{eq:lowerdev}. This implies $\mathbf{E}p^{S_k} \leq e^{-c\sqrt{k}}$ for some $c>0$.
	Moreover, since $\mathbf{E}S_k=O(\sqrt{k}\log k)$ by Proposition~\ref{prop:expectSk}, we have $\mathbf{E}p^{S_k} \geq e^{-c'\sqrt{k}\log k}$.
This proves~\eqref{eq:asympt}.

	To obtain~\eqref{eq:asympt2}, use the Markov inequality for the $k$-th moment:
	$$
	\mathbf{P}(\Lambda_p>1-\epsilon)= \mathbf{P}(\Lambda_p^k>(1-\epsilon)^k) \leq (1-\epsilon)^{-k} a_k(p)\leq \exp\big(-c'\sqrt{k} -k \log(1-\epsilon)\big),
	$$
	where we used the right inequality in~\eqref{eq:asympt}. Taking $k=c \epsilon^{-2}$ for small enough $c$ gives 
	$\mathbf{P}(\Lambda_p>1-\epsilon) \leq \exp\big(-c''\epsilon^{-1}\big)$.

\end{proof}

\bibliographystyle{alpha}
\bibliography{biblio}

\end{document}